\documentclass[a4paper]{amsart}
\usepackage{amssymb,amsmath,amsfonts,amsthm,mathtools}
\usepackage{enumerate}
\usepackage[protrusion=true]{microtype}
\usepackage[T1]{fontenc}
\usepackage[english]{babel}
\usepackage{xcolor,colortbl}
\usepackage{hyperref}
\newtheorem{theorem}{Theorem}[section]
\newtheorem{lemma}[theorem]{Lemma}
\newtheorem{proposition}[theorem]{Proposition}
\newtheorem{corollary}[theorem]{Corollary}

\newtheorem{conconjecture}[theorem]{Confirmed conjecture}
\theoremstyle{definition}
\newtheorem{model}{Assumption}
\newenvironment{assumption}[1]{
	\renewcommand\themodel{#1}
	\model
}{\endmodel}
\theoremstyle{remark}
\newtheorem{remark}[theorem]{Remark}
\numberwithin{equation}{section}
\newcommand{\RR}{\mathbb{R}}
\newcommand{\CC}{\mathbb{C}}
\newcommand{\NN}{\mathbb{N}}
\newcommand{\ZZ}{\mathbb{Z}}

\newcommand{\cF}{\mathcal{F}}
\newcommand{\cD}{\mathcal{D}}

\newcommand{\cK}{\mathcal{K}}

\newcommand{\eps}{\varepsilon}
\newcommand{\euler}{\mathrm{e}}

\newcommand{\fh}{\mathfrak{h}}

\newcommand{\grad}{\nabla}
\newcommand{\fa}{{\mathfrak a}}

\newcommand{\fv}{{\mathfrak v}}

\newcommand{\indic}{\mathbf{1}}

\DeclareMathOperator{\Ran}{Ran}

\newcommand*\Diff[1]{\mathop{}\!\mathrm{d}#1}
\let\Re\relax
\DeclareMathOperator{\Re}{Re}
\let\Im\relax
\DeclareMathOperator{\Im}{Im}

\newcommand{\loc}{\mathrm{loc}}
\newcommand{\obs}{\mathrm{obs}}
\DeclarePairedDelimiter{\abs}{|}{|}
\DeclarePairedDelimiter{\norm}{\lVert}{\rVert}
\newcommand{\bes}{\begin{equation*}}
\newcommand{\ees}{\end{equation*}}
\newcommand{\be}{\begin{equation}}
\newcommand{\ee}{\end{equation}}
\newcommand{\eqs}[1]{\begin{align*}#1\end{align*}}

\definecolor{darkred}{rgb}{0.5,0,0}
\definecolor{darkgreen}{rgb}{0,0.5,0}
\definecolor{darkblue}{rgb}{0,0,0.5}
\title[Spectral inequality for power growth potentials]
{Spectral inequality with sensor sets of decaying density for Schr\"odinger operators with power growth potentials}
\subjclass[2010]{Primary 35Pxx; Secondary 35J10, 35B40.}
\keywords{Spectral inequalities, uncertainty relation, Schroedinger operator, Shubin operator, decay of eigenfunctions,
confinement potential, observability.}
\hypersetup{
	pdftitle={Spectral inequality with sensor sets of decaying density for Schr\"odinger operators with power growth potentials},
	pdfauthor={Alexander Dicke, Albrecht Seelmann, Ivan Veseli\'c},
	pdfkeywords={Spectral inequalities, uncertainty relation, Schroedinger operator, Shubin operator, decay of eigenfunctions, confinement potential, observability},
	colorlinks,
	linkcolor=darkblue,
	filecolor=darkgreen,
	urlcolor=darkred,
	citecolor=darkblue,
}
\author[A.~Dicke]{Alexander Dicke}
\author[A.~Seelmann]{Albrecht Seelmann}
\author[I.~Veseli\'c]{Ivan Veseli\'c}
\address[A.D., A.S., I.V.]{
	Technische Univer\-si\-t\"at Dortmund,
	Germany
}
\urladdr{\url{https://www.mathematik.tu-dortmund.de/lsix/research/analysis/}}
\email{adicke.math@gmail.com, \{albrecht.seelmann,ivan.veselic\}@tu-dortmund.de}
\begin{document}
%
%
\begin{abstract}
	We prove a spectral inequality (a specific type of uncertainty relation) for Schr\"odinger operators with confinement potentials,
	in particular of Shubin-type. The sensor sets are allowed to decay exponentially,
	where the precise allowed decay rate depends on the potential.
	The proof uses an interpolation inequality derived by Carleman estimates, quantitative
	weighted $L^2$-estimates and an $H^1$-concentration estimate, all of them for functions in a spectral subspace of the operator.
\end{abstract}
\maketitle
%
%
\section{Introduction, main results, and discussion}
In the context of control theory, a \emph{spectral inequality} for a nonnegative
selfadjoint operator $H$ in $L^2(\RR^d)$ is an inequality of the form
\be\label{eq:abstract-spectral-inequality}
	\norm{f}_{L^2(\RR^d)}
	\leq
	d_0\euler^{d_1 \lambda^s}\norm{f}_{L^2(\omega)}
	\quad
	\text{for all}
	\quad
	f \in \Ran P_\lambda(H)
	,
	\
	\lambda \geq 1
	,
\ee
with a measurable set $\omega \subset \RR^d$ and some $s \in (0,1)$ and $d_0,d_1 > 0$.
Having the application in control theory in mind, we will refer to $\omega$ as a sensor set.
Here, $\lambda\mapsto P_\lambda(H) = \indic_{(-\infty,\lambda]}(H)$ denotes the resolution of identity associated to $H$.
Once Inequality~\eqref{eq:abstract-spectral-inequality} is at disposal, the famous Lebeau-Robbiano method \cite{LebeauR-95},
see also \cite{TenenbaumT-11,BeauchardPS-18,NakicTTV-20, GallaunST-20}, allows to conclude
(final state) observability and null-controllability
from $\omega$ of the abstract Cauchy problem associated to $-H$.
Note that the spectral inequality is a manifestation of the uncertainty principle
and has been used also in other contexts albeit with a different name, see, e.g., \cite{EgidiNSTTV-20}.

The new results in this paper concern Schr\"odinger operators $H = -\Delta + V$
with potentials $V$ growing at infinity.
A paradigmatic model in this context is the harmonic oscillator $H = -\Delta + \abs{x}^2$,
for which spectral inequalities were proven in \cite{BeauchardPS-18,BeauchardJPS-21,EgidiS-21,MartinPS-22,DickeSV-23}.
See also \cite{DickeSV-22} for a treatment of partial harmonic oscillators.
In particular, \cite{DickeSV-23} shows that \eqref{eq:abstract-spectral-inequality}
holds for $H = -\Delta + \abs{x}^2$ with measurable sensor sets $\omega \subset \RR^d$
satisfying
\be\label{eq:thick-decay}
	\frac{\abs{\omega\cap (k+(-\rho/2,\rho/2)^d)}}{\rho^d}
	\geq
	\gamma^{1+\abs{k}^\alpha}
	\quad
	\text{for all}
	\quad
	k \in \ZZ^d
	,
\ee
with some parameters $\rho > 0$, $\gamma \in (0,1]$, and $\alpha \in [0,1)$.

Sets satisfying \eqref{eq:thick-decay} with $\alpha = 0$ are called \emph{thick sets}.
They are crucial in two aspects:
First, for the pure Laplacian $H = -\Delta$ an inequality like \eqref{eq:abstract-spectral-inequality}
holds \emph{if and only if} $\omega$ is a thick set. This has been studied in
\cite{Panejah-61,Panejah-62,Kacnelson-73,LogvinenkoS-74}
culminating in sharp inequalities established in \cite{Kovrijkine-thesis,Kovrijkine-01}.
These results are in the literature often referred under the term `Logvinenko--Sereda inequality'.
Second, as shown in \cite{EgidiV-18} and \cite{WangWZZ-19} independently,
the heat equation on $\RR^d$ is observable (hence final state controllable) from a set $\omega\subset \RR^d$
\emph{if and only if} $\omega$ is a thick set.
This illustrates that this type of set is a benchmark for comparison
for the topics at hand.

For a Schr\"odinger operator $H = -\Delta + V$ with a bounded, real-valued, suitably analytic potential $V$
vanishing at infinity, a spectral inequality with thick sets $\omega$ was shown in \cite{LebeauM-19}.
Thick sets are also relevant for spectral and parabolic observability inequalities on bounded domains,
mostly if one considers a sequence of them like $(-L,L)^d$ and aims for bounds uniform in $L\in\NN$
cf.~\cite{EgidiV-20} and \cite{SeelmannV-20}.
See also \cite{LogunovM-20,BurqM-23} for related results for divergence type operators.

Although we conjecture that a condition similar to \eqref{eq:thick-decay} would be appropriate in the context of this paper,
see Section \ref{sec:conjectures} for a more precise statement, due to the methods applied, we need to assume more regularity of the sensor set.
More precisely, we have to work in the category of open rather than measurable sets.

In this setting the natural analog of thick sets is the following: For $0 < \delta < 1/2$,
a measurable $\omega$ is called a \emph{$\delta$-equidistributed set},
if each intersection $\omega \cap (k + (-1/2,1/2)^d)$, $k \in \ZZ^d$, contains a ball of radius $\delta$.
This term was introduced in \cite{RojasMolinaV-13} (but see also \cite{Germinet-08,GerminetK-13} for related notions)
in the context of random Schr\"odinger operators with bounded potentials
and shown to be sufficient for a weaker form of uncertainty relation than \eqref{eq:abstract-spectral-inequality},
where $f$ needs to be an individual eigenfunction of $H$, rather than an element of a spectral subspace $\Ran P_\lambda(H)$.
This condition on $f$ was subsequently relaxed in \cite{Klein-13} to allow for $f \in  \Ran \indic_{(\lambda-\eps,\lambda]}(H) $,
for specified, but small $\eps>0$. In \cite{KleinT-16} the argument was extended to potentials with uniformly controlled local singularities.
Finally, a proper spectral inequality for Schr\"odinger operators with bounded potentials of the form
\eqref{eq:abstract-spectral-inequality} was obtained in \cite{NakicTTV-18,NakicTTV-20a},
the difference being that the first paper treats domains of the type $(-L,L)^d$, whereas the second allows also
certain `rectangular' unbounded domains, as well,  including $\RR^d$ itself.
Again, this result was extended in \cite{DickeRST-20} to potentials with certain local singularities, giving an improvement of \cite{KleinT-16}, as well.
A common feature of these papers is that they trace the dependence of uncertainty relation on the geometry and model parameters,
in contrast to a many earlier papers with qualitative results. In particular, \cite{NakicTTV-20} extends the techniques of the ingenious paper
\cite{JerisonL-99} (itself based on the seminal \cite{DonnellyF-88}) from compact domains to $\RR^d$  using a geometric covering construction and quantifying the above mentioned parameter dependence.

Let us come back to the harmonic oscillator.
For the harmonic oscillator on the domain $\RR^d$ thick set have been shown in \cite{BeauchardJPS-21} to be sufficient for a spectral inequality.
In \cite{EgidiS-21} this was generalized to certain unbounded domains exhibiting sufficient symmetry.
However, in contrast to the pure Laplacian, this condition is not optimal here.
In fact, the fast growth of the quadratic potential yields fast decay of eigenfunctions of the harmonic oscillator, i.\,e., Hermite functions.
This explains why one can weaken the geometric condition on the sensor sets to a decay of the form \eqref{eq:thick-decay}.
In particular, \eqref{eq:thick-decay} allows the sensor set to have finite Lebesgue measure, see,
e.\,g., the set $\omega$ considered in \eqref{eq:finite-measure} below.

One naturally expects similar phenomena for Schr\"odinger operators of the
form $H = -\Delta + \abs{x}^\tau$ with general $\tau > 0$,
in particular that the unbounded potential enforces fast decay of the eigenfunctions.
This should make it possible to allow a similar (or even faster) decay in the sensor set as in \eqref{eq:thick-decay} above,
while still obtaining a spectral inequality of the form \eqref{eq:abstract-spectral-inequality} with explicit dependence on the
geometry of $\omega$.

An important step in this direction was achieved in \cite{Miller-08}, who proves a spectral inequality for $\tau\in2\NN$ with the sensor set $\omega$ being a cone of the form
\be\label{eq:cone}
	\{ x \in \RR^d \colon \abs{x} \geq r_0\text{ and } x/\abs{x} \in \Omega_0\}
	\quad
	\text{for some open}
	\quad
	\Omega_0 \subset \mathbb{S}^{d-1}
	.
\ee
Note that the cone has infinite Lebesgue measure in contrast to typical sets satisfying
\eqref{eq:thick-decay} with $\alpha>0$.

The operator considered in \cite{Miller-08} is not only a Schr\"{o}dinger operator, but at the same time a
Shubin operator, see Section \ref{sec:conjectures} for a definition and discussion of the latter class.
In contrast, the methods of this paper do not require any integer condition on the power $\tau$ of the potential.

More precisely, our main new result is a spectral inequality for
Schr\"odinger operators $H = -\Delta + V$ with potentials $V$ of the following type.

\begin{assumption}{A}\label{ass:potential}
	Suppose that $V \in W^{1,\infty}_\loc(\RR^d)$ is such that
	\begin{enumerate}[(i)]
		\item for some $c_1, c_2 > 0$ and some $0 < \tau_1 \leq \tau_2$
		we have $c_1\abs{x}^{\tau_1} \leq V(x) \leq c_2\abs{x}^{\tau_2}$ for all
		$x \in \RR^d$;
		\item for some $\nu > 0$ we have
		\be\label{eq:Vgrowth}
			M_\nu
			:=
			\norm{\euler^{-\nu\abs{x}}\abs{\grad V}}_{L^\infty(\RR^d\setminus B(0,1))}
			<
			\infty
			.
		\ee
	\end{enumerate}
\end{assumption}

\begin{theorem}[Spectral inequality]\label{thm:spectral-inequality}
	Let $H = -\Delta + V$ with $V$ as in Assumption~\ref{ass:potential}, and let $\omega \subset \RR^d$ be measurable such that for
	some $\delta \in (0,1/2)$ and $\alpha \geq 0$ each intersection $\omega \cap (k + (-1/2,1/2)^d)$, $k \in \ZZ^d$, contains a ball
	of radius $\delta^{1+\abs{k}^\alpha}$.
	Then there is a constant $C > 0$ depending only on $\tau_1, \tau_2, c_1, c_2, \nu, M_\nu$,
	and the dimension $d$ such that
	\be\label{eq:spectral-inequality}
		\norm{f}_{L^2(\RR^d)}
		\leq
		\Bigl(\frac{1}{\delta}\Bigr)^{C^{1+\alpha}\cdot \lambda^{(\alpha+2\tau_2/3)/\tau_1}}
		\norm{f}_{L^2(\omega)}
	\ee
	for all $f \in \Ran P_\lambda(H)$, $\lambda \geq 1$.
\end{theorem}

In the case $\alpha = 0$, the sets $\omega$ considered in Theorem~\ref{thm:spectral-inequality} correspond to
\emph{$\delta$-equidistributed} sets defined above.
More interesting and enlightening is the set
\be\label{eq:finite-measure}
	\omega
	=
	\bigcup_{k \in \ZZ^d} B(k,2^{-(1+\abs{k}^\alpha)})
	,
\ee
which has finite measure if $\alpha > 0$.
Moreover, if $\alpha < 1$, $\omega$ satisfies \eqref{eq:thick-decay} with $\rho = 1$ and suitably chosen $\gamma \in (0,1)$.

Note that inequality \eqref{eq:spectral-inequality}
is of the form \eqref{eq:abstract-spectral-inequality},
with the exponent $s$ satisfying $s = (\alpha+2\tau_2/3)/\tau_1 < 1$ if and only if $0 \leq \alpha < \tau_1-2\tau_2/3$.
Hence, if Assumption~\ref{ass:potential} holds and additionally $2\tau_2 < 3\tau_1$,
Theorem~\ref{thm:spectral-inequality} directly
leads to observability of the abstract Cauchy problem associated to $-H$.
More precisely, we get the following result using the variant of the Lebeau-Robbiano method spelled out in
\cite[Theorem~2.8]{NakicTTV-20}.

\begin{corollary}[Final state observability inequality]\label{cor:observability}
	Let $V$ be as in Assumption~\ref{ass:potential}, let $(e^{t(\Delta-V)})_{t\geq 0}$ be the $C_0$-semigroup with generator $\Delta - V$,
	and let $\omega$ be as in Theorem~\ref{thm:spectral-inequality} with
	$0\leq \alpha < \tau_1-2\tau_2/3$.
	
	Then, for all $T > 0$ and all $g\in L^2(\RR^d)$ we have
	\be\label{eq:observability}
		\norm{e^{T(\Delta-V)}g}_{L^2(\RR^d)}^2
		\leq
		C_\obs^2
		\int_0^T \norm{e^{t(\Delta-V)}g}_{L^2(\omega)}^2\Diff{t}
		.
	\ee
	The constant satisfies the asymptotics $C_\obs = C_\obs(T) = \mathcal{O}(T^{-1/2})$ as $T \to \infty$ and  the bound
\be\label{eq:Cobs}
	C_\obs^2
	\leq
	\frac{K}{T} (2d_0+1)^K \exp\Biggl[ K \left(\frac{d_1}{T^s}\right)^{1/(1-s)} \Biggr]
\ee
with $s=\frac{\alpha+2\tau_2/3}{\tau_1}$,
$d_1=-C^{1+\alpha}\ln\delta$, $d_0=\delta^{-C^{1+\alpha}}$,
$C=C(\tau_1, \tau_2, c_1, c_2, \nu, M_\nu,d)>0$ depending only on $\tau_1, \tau_2, c_1, c_2, \nu, M_\nu$, and the dimension $d$,
and $K=K(s)\geq 1$ depending only on $s$.
\end{corollary}

\begin{remark}[Bound on the observability constant]\label{rem:observability_constant}
For general potentials satisfying Assumption~\ref{ass:potential} the bound \eqref{eq:Cobs} seems rather complicated.
Let us concentrate on the question, how $C_\obs$ depends on the time parameter $T$ and specialize to potentials of the form
$V(x)=\abs{x}^\tau$ (for some $\tau >3\alpha$). The bound \eqref{eq:Cobs} on the observability constant $C_\obs$ in \eqref{eq:observability}
can then be rewritten as
	\[
		C_\obs^2 \leq \frac{D}{T} \exp\Biggl(\frac{D}{T^{1+ \frac{2\alpha+\tau/3}{\tau/3-\alpha}}} \Biggr)
		\quad\text{for all}\quad T > 0
	\]
	where $D > 0$ is a constant that depends only on $\tau, \alpha$, $\delta$, and the dimension.
\end{remark}

Let us briefly comment on the hypotheses in Assumption~\ref{ass:potential}:
The lower bound in part~(i) allows to
bound the eigenvalue counting function for $H$, cf.\ \eqref{eq:boundN} below.
It also implies
a suitable $L^2$-decay for eigenfunctions of $H$,
see Proposition~\ref{prop:weighted-est-eigenfunction} below.
The bound in part~(ii) allows to obtain a similar decay for partial
derivatives of eigenfunctions by differentiating the eigenvalue equation $Hf = \lambda f$, which introduces partial derivatives
of the potential to the equation, see Proposition~\ref{prop:weighted-est-gradient-eigenfunction} below. Together with the bound
on the eigenvalue counting function, this amounts to the fact that the $H^1$-mass of $f\in \Ran P_\lambda(H)$ is strongly
localized. This is made precise in the following theorem.

\begin{theorem}\label{thm:decay-intro}
	Let $H = -\Delta + V$ with $V$ as in Assumption~\ref{ass:potential}.
	Then there is a constant $C' > 0$, depending only on $\tau_1,c_1,\nu, M_\nu$, and the dimension $d$,
	such that for all $\lambda \geq 1$ and all $f \in \Ran P_\lambda(H)$ we have
	\be\label{eq:localization-confinement}
		\norm{f}_{H^1(\RR^d\setminus B(0,C'\lambda^{1/\tau_1}))}^2
		\leq
		\frac{1}{2}\norm{f}_{L^2(\RR^d)}^2
		.
	\ee
\end{theorem}

This result enables us essentially to reduce the considerations to a sufficiently large cube $\Lambda= (-L,L)^d$ in $\RR^d$
by a `cut-off procedure' (cf.~Remark~\ref{rem:Carleman} below).
At this point the upper bound in part~(i) of Assumption~\ref{ass:potential} comes into play and ensures an effective bound on $V$.
More precisely, $\|V\|_{\infty,\Lambda}:=\sup_{x \in\Lambda} V(x) \leq  c_2 d^{\tau_2/2} \, L^{\tau_2}$.
\medskip

We now discuss under which conditions on the parameter $\alpha$ (encoding the decay rate)
the spectral inequality provided by Theorem~\ref{thm:spectral-inequality}
is of the type \eqref{eq:abstract-spectral-inequality} and hence usable as input for the observability inequality
in Corollary~\ref{cor:observability}.

To this end, let us focus on the harmonic oscillator, that is, $H = -\Delta + V$ with $V(x) = \abs{x}^2$.
This operator meets the hypotheses of Theorem~\ref{thm:spectral-inequality} (with $\tau_1 = \tau_2 = 2$, $c_1 = c_2 = 1$).
Furthermore, condition $0\leq \alpha < \tau_1-2\tau_2/3$ required for $s \in (0,1)$ in  \eqref{eq:abstract-spectral-inequality},
simplifies to $0\leq \alpha < 2-4/3=2/3$.
By contrast, \cite{DickeSV-23} derives Inequality~\eqref{eq:abstract-spectral-inequality} merely under the condition $\alpha < 1$.
Let us further restrict our considerations to the case $\alpha=0$, i.\,e.~that $\omega$ is a thick set.
Then the exponent on the right-hand side of \eqref{eq:spectral-inequality} reads
$C \lambda^{2/3}$, while the corresponding term in \cite{DickeSV-23} is $C \lambda^{1/2}$
which is the expected behaviour, see for instance the discussion on page~14 in \cite{Kovrijkine-thesis} or in Section~3 of \cite{EgidiV-20}.
The slightly worse behaviour in our Theorem~\ref{thm:spectral-inequality} above is
due to the application of a Carleman estimate after the mentioned `cut-off procedure' to the cube $\Lambda= (-L,L)^d$.
While according to a conjecture associated with the name of Landis, the sup norm of the potential should
enter the unique continuation estimate in the scaling $\|V\|_{\infty,\Lambda}^{1/2}$
the current technology based on Carleman estimates yields the scaling $\|V\|_{\infty,\Lambda}^{2/3}$.
Tracing this suboptimality further leads to the suboptimal condition on the exponent $\alpha$.
See also the discussion in the next section.
\medskip

Let us finally mention that our results can also be extended to potentials $V\colon \RR^d = \RR^{d_1}\times\RR^{d_2} \to [0,\infty)$
which grow only in certain coordinate directions.
For simplicity, we here restrict ourselves to potentials $V(x) = \abs{x_1}^\tau$ with $x = (x_1,x_2) \in \RR^{d_1}\times\RR^{d_2}$ and $\tau > 0$.
The more general analogue to Theorem~\ref{thm:spectral-inequality} is presented in Theorem~\ref{thm:anisotropic_V} below.

\begin{theorem}\label{thm:anisotropic_V_easy}
	Let $H = -\Delta + \abs{x_1}^\tau$, and let $\omega \in \RR^d$ be measurable such that for some
	$\delta \in (0,1/2)$ and $\alpha \geq 0$ each intersection $\omega \cap (k + (-1/2,1/2)^d)$, $k = (k_1,k_2) \in \ZZ^d$, contains
	a ball of radius $\delta^{1+\abs{k_1}^\alpha}$.
	Then there is a constant $C > 0$ depending only on the dimension $d$ and the parameter $\tau$
	such that for all $\lambda \geq 1$ and all $f \in \Ran P_\lambda(H)$ we have
	\bes
		\norm{f}_{L^2(\RR^d)}
		\leq
		\Bigl(\frac{1}{\delta}\Bigr)^{C^{1+\alpha}\cdot \lambda^{\alpha/\tau+2/3}}
		\norm{f}_{L^2(\omega)}
		.
	\ees
\end{theorem}

The rest of this note is organized as follows.
The purpose of Section~\ref{sec:conjectures} is threefold:
We consider some related models, in particular Shubin operators,
formulate conjectures, that have been spelled out in the first version of this manuscript, and finally discuss papers of other authors, which have been
obtained during the refereeing procedure of this manuscript,
actually confirming the conjectures.
Based on \cite{GagelmanY-12}, Section~\ref{sec:decay} discusses decay
properties of eigenfunctions and establishes the $H^1$-localization for elements $f \in \Ran P_\lambda(H)$,
i.\,e.~Theorem~\ref{thm:decay-intro}.
Section~\ref{sec:proof} then revisits the proofs of the spectral
inequalities developed in \cite{NakicTTV-18} and improved in \cite{NakicTTV-20a,DickeRST-20} and adapts them towards a proof of
Theorem~\ref{thm:spectral-inequality}.
Finally, in Section~\ref{sec:tensor} we discuss how to extend our results to partial confinement potentials
based on a tensorization technique.

In all these considerations, we frequently write $A \lesssim B$ with quantities $A$ and
$B$ if there is a constant $c > 0$ depending on the model parameters such that $A \leq c B$. If the constant depends only
on a subset of the model parameters, we occasionally write these parameters in the subscript of $\lesssim$, for instance,
$A \lesssim_d B$ if the constant only depends on the dimension $d$.

\subsection*{Acknowledgments}
A.D. and A.S. have been partially supported by the DFG grant VE 253/10-1 entitled
\emph{Quantitative unique continuation properties of elliptic PDEs with variable
2nd order coefficients and applications in control theory, Anderson localization,
and photonics}.

\section{Conjectures, generalizations and further developments} \label{sec:conjectures}
In this section, we discuss some related models,
in particular Shubin operators, and formulate some conjectures
that were spelled out in our original preprint.
In fact, recent papers obtain more general results than ours and mostly confirm our original conjectures,
in some cases quoting them as motivation.
Let us discuss this properly.

We expected that a (partial) improvement of Theorem~\ref{thm:spectral-inequality} of the following form is valid.

\begin{conconjecture}\label{conj:general-power}
	Let $\tau > 0$, and suppose that $\omega \subset\RR^d$ is measurable such that for
	some $\delta \in (0,1/2)$ and $\alpha \geq 0$ each intersection $\omega \cap (k + (-1/2,1/2)^d)$, $k \in \ZZ^d$, contains a ball
	of radius $\delta^{1+\abs{k}^\alpha}$.
	Then there is a constant $C > 0$ depending only on
	$\tau$, $\delta$, $\alpha$, and the dimension $d$ such that for all $\lambda \geq 1$ and all $f \in \Ran P_\lambda(-\Delta+\abs{x}^\tau)$ we have
	\bes
		\norm{f}_{L^2(\RR^d)}
		\leq
		\Bigl(\frac{1}{\delta}\Bigr)^{C\cdot \lambda^{\frac{\alpha}{\tau}+\frac{1}{2}}}
		\norm{f}_{L^2(\omega)}
		.
	\ees
\end{conconjecture}

\newpage

In fact, recently \cite{ZhuZ-23} confirmed the conjecture considering even more general potentials of the following type:

\begin{assumption}{B}\label{ass:ZZ}
	Suppose that $V= V_1+V_2$ with   $V_1 \in W^{1,\infty}_{\mathrm{loc}}(\RR^d)$ and $V_2 \in L^{\infty}_{\mathrm{loc}}(\RR^d)$
satisfies
	\begin{enumerate}[(i)]
      \item for some $\tau_1, c_1> 0$ we have
    \[
    c_1(\abs{x}-1)_+^{\tau_1} \leq V(x) \quad \text{  for all } x \in \RR^d \text{ and }
    \]
		\item for some $\tau_2, c_2> 0$ we have
		\[
       |V_1(x)|+ |\nabla V_1(x)|+ |V_2(x)|^{4/3}\leq c_2(\abs{x}+1)^{\tau_2}  \quad \text{  for all } x \in \RR^d.
		\]
	\end{enumerate}
\end{assumption}

\begin{theorem}[Theorem~1 in \cite{ZhuZ-23}]\label{thm:spectral-inequality-ZZ}
	Let $H = -\Delta + V$ with $V$ as in Assumption~\ref{ass:ZZ}, and let $\omega \subset \RR^d$ be measurable such that for
	some $\delta \in (0,1/2)$ and $\alpha \geq 0$ each intersection $\omega \cap (k + (-1/2,1/2)^d)$, $k \in \ZZ^d$, contains a ball
	of radius $\delta^{1+\abs{k}^\alpha}$.
	Then there is a constant $C > 0$ depending only on $\tau_1, \tau_2, c_1, c_2,\alpha$, 	and the dimension $d$ such that
	\[
		\norm{f}_{L^2(\RR^d)}
		\leq
		\Bigl(\frac{1}{\delta}\Bigr)^{C\cdot \lambda^{(\alpha/\tau_1) +(\tau_2/2\tau_1)}}
		\norm{f}_{L^2(\omega)}
	\]
	for all $f \in \Ran P_\lambda(H)$.
\end{theorem}

The proof of \cite{ZhuZ-23} follows the same line of argument as ours but develops several crucial technical improvements,
among them a Carleman estimate using a bound on the gradient of the potential, not just the sup norm of the potential itself,
and Cacciopoli inequalities.
We recommend the reader interested in the currently best obtainable estimates in our context to consult their proofs.

\cite{ZhuZ-23} in turn motivated the study of null-controllability of the heat equation associated to
fractional Baouendi-Grushin operators in \cite{JamingW-23}.
The paper contains several related results. We quote a particular case of one of them as an example that considers potentials of the following class:

\begin{assumption}{C}\label{ass:JW}
Suppose that $V= V_1+V_2$ with   $V_1 \in W^{1,\infty}_{\mathrm{loc}}(\RR^d)$ and $V_2 \in L^{\infty}_{\mathrm{loc}}(\RR^d)$
satisfies for some $\tau, c_1, c_2> 0$
	\begin{enumerate}[(i)]
      \item
    \[
    c_1\abs{x}^{\tau} \leq V(x) \quad \text{  for all } x \in \RR^d \text{ and }
    \]
		\item
		\[
       |V_1(x)|+ |\nabla V_1(x)|+ |V_2(x)|^{4/3}\leq c_2(\abs{x}+1)^{\tau}  \quad \text{  for all } x \in \RR^d.
		\]
	\end{enumerate}
\end{assumption}

\begin{theorem}[Cf.~Theorem~1.7 in \cite{JamingW-23}]\label{thm:null-contollability-JW}
Let $\delta \in (0,1/2)$ and $V$ be as in Assumption~\ref{ass:JW}.
Set $\theta_*:= \frac{\tau+2}{3}$.
Let $\omega \subset \RR^d$ be a $\delta$-equidistributed set.
Then for every initial condition $\phi_0 \in L^2(\RR^d\times \RR^m)$ the problem
\begin{align*}
\partial_t \phi(t,x,y) + \left( -\Delta_x-V(x)\Delta_y\right)^\theta \phi(t,x,y)
&=   (\indic_{\omega \times \RR^m} \cdot u)(t,x,y),
&t >0, x \in \RR^d, y \in \RR^m,
\\
 \phi(0,x,y) &=  \phi_0(x,y) , &x \in \RR^d, y \in \RR^m.
\end{align*}
is final state null-controllable in every time $T>0$, if $\theta>\theta_*$.
Furthermore, there exists a time $T_* >0$ such that
for $\theta=\theta_*$ the above problem is final state null-controllable
provided $T>T_*$. For $\delta \searrow 0$, the time $T_*$ diverges like $\ln\left(\frac{1}{\delta}\right)$.
\end{theorem}

\bigskip

Due to the use of Carleman inequalities, we can prove the spectral inequality in Theorem~\ref{thm:spectral-inequality} only
for sensor sets containing suitable open balls, but not for measurable sets of the form \eqref{eq:thick-decay}.
This prompts the question, how and under what conditions this restriction can be removed.
Natural candidates are polynomial or analytic potentials, in particular in view of  \cite{LebeauM-19}.

In this vein, let us discuss next results on Shubin operators, that is, Hamiltonians of the form $H = (-\Delta)^m + \abs{x}^{2l}$
with $m,l \in \NN$. Note that the particular case of $m=l=1$ here agrees with the harmonic oscillator. For $l > m = 1$, a first
proof of a spectral inequality of the form \eqref{eq:abstract-spectral-inequality}, and thus observability of
the system generated by $-H$, was established in \cite{Miller-08}, but only for sensor sets $\omega$ as in \eqref{eq:cone}.
Arbitrary sets $\omega$ with strictly positive measure were treated recently in \cite{Martin-22}. There, the author verifies that
for such sets and general $m,l \in \NN$ we have the inequality
\bes
	\norm{f}_{L^2(\RR^d)}
	\leq
	d_0\euler^{d_1 \lambda^{\frac{1}{2m} + \frac{1}{2l}}\log\lambda}\norm{f}_{L^2(\omega)}
	\quad
	\text{for all}
	\quad
	f \in \Ran P_\lambda(H)
	,
	\
	\lambda \geq 1
	,
\ees
which in the case of $\max\{m,l\} > 1$ is indeed of the form \eqref{eq:abstract-spectral-inequality} with the choice
$s = 1/(2m) + 1/(2l) + \eps < 1$ for $\eps \in (0, 1/4)$. The constants $d_0$ and $d_1$ are, however, not explicit in
this case.

A similar result in \cite{Martin-22} treated the case of sensor sets $\omega$ that are thick with respect to certain unbounded
scales, which allows for holes in $\omega$ of growing size. This takes into account more precise information on $\omega$ and
results in an improved dependence on $\lambda$, but the sensor sets are required to have infinite Lebesgue measure and the
constants $d_0$ and $d_1$ are still not explicit.

However, in light of the Bernstein inequalities proved in \cite[Proposition~4.1]{Martin-22}
we expected that the approach presented in \cite{DickeSV-23} could be adapted
for Shubin operators $H = (-\Delta)^m + \abs{x}^{2l}$, $m,l \in \NN$ in order
to obtain a result of the following form:

\begin{conconjecture}\label{conj:Shubin}
	Let $m,l\in\NN$, and suppose that $\omega \subset\RR^d$ satisfies \eqref{eq:thick-decay} with some $\rho > 0$,
	$\gamma \in (0,1]$, and $0 \leq \alpha < l$. Then there are constants $K,C > 0$, with $K$ depending only on $m$, $l$, and the
	dimension $d$, and $C$ depending additionally also on $\rho$ and $\alpha$, such that for all $\lambda \geq 1$ and all
	$f \in \Ran P_\lambda((-\Delta)^m+\abs{x}^{2l})$ we have
	\be\label{eq:conjSpecIneq}
		\norm{f}_{L^2(\RR^d)}
		\leq
		\Bigl(\frac{K}{\gamma}\Bigr)^{C\cdot \lambda^{\frac{\alpha}{2l}+\frac{1}{2m}}}
		\norm{f}_{L^2(\omega)}
		.
	\ee
\end{conconjecture}

This conjecture was indeed confirmed recently in \cite{AlphonseS-22} by the second author and P.~Alphonse. In fact, while
\cite{Martin-22} relies on smoothing estimates for the semigroup generated by $-H$ established in \cite{Alphonse-20b},
\cite{AlphonseS-22} directly works with the underlying Agmon estimates from \cite{Alphonse-20b} in order to prove a variant of
the Bernstein inequalities that allow for very explicit spectral inequalities under fairly general assumptions on the sensor set.
In particular, the estimates of \cite{AlphonseS-22} confirm \eqref{eq:conjSpecIneq} with the explicit bound on the constant
\[
	C
	\leq
	\tilde{K}^{1+\alpha} (1 + \rho^{1+\frac{k}{m}} + \rho)
	,
\]
with $\tilde{K} > 0$ depending only on $m$, $l$, and the dimension $d$, see Theorem~2.3 and Remark~2.5 in \cite{AlphonseS-22}.
Note that the general case of \cite{AlphonseS-22} actually treats sensor sets with a more general geometry than \eqref{eq:thick-decay}.

\section{Decay of eigenfunctions}\label{sec:decay}

In this section we quantify decay properties of linear combinations of eigenfunctions for
the operator $H = -\Delta + V$ with $V$ as in Assumption~\ref{ass:potential}. Although there are
several results available for eigenfunctions establishing a fast decay in $L^2$-sense, see,
e.g., \cite{Simon-75,Agmon-82,Davies-82,DaviesS-84,BerezinS-91}, we need an explicit weighted $L^2$-estimate also
for the partial derivatives of first order. The approach in \cite{GagelmanY-12} seems to be
the most convenient one for this task. However, since it is essential for us to have the
dependence of the decay on the spectral parameter explicitly quantified, we have to revisit the
reasoning from \cite{GagelmanY-12} and extract the statements we need.

\subsection{Properties of the operator and main objective}\label{ssec:operator}
We begin with a review of the construction of the operator $H = -\Delta + V$ with $V$ as in
Assumption~\ref{ass:potential} and a collection of its basic properties.

Consider the forms
\[
	\fa[ f , g ]
	:=
	\int_{\RR^d} \nabla f(x) \cdot \nabla g(x) \Diff{x}
	,\quad
	f,g \in \cD[\fa] := H^1(\RR^d)
	,
\]
as well as $\cD[\fv] := \{ f \in L^2(\RR^d) \colon V^{1/2}f \in L^2(\RR^d) \}$,
\[
	\fv[ f , g ]
	:=
	\langle V^{1/2}f , V^{1/2}g \rangle_{L^2(\RR^d)}
	,\quad
	f,g \in \cD[\fv]
	,
\]
and
\[
	\fh
	:=
	\fa + \fv
	,\quad
	\cD[\fh]
	:=
	\cD[\fa] \cap \cD[\fv]
	.
\]
The nonnegative form $\fa$ is closed since $H^1(\RR^d)$ is complete, and $\fv$ is nonnegative and closed by
\cite[Proposition~10.5\,(ii)]{Schmuedgen-12}. Thus, the form $\fh$ is densely defined, nonnegative, and closed by
\cite[Corollary~10.2]{Schmuedgen-12}, so that there is a unique (nonnegative) selfadjoint operator $H$ on $L^2(\RR^d)$ given by
\[
	\cD(H)
	=
	\{ f \in \cD[\fh] \colon \exists h \in L^2(\RR^d) \text{ s.t. }
		\fh[ f , g ] = \langle h , g \rangle_{L^2(\RR^d)}\ \forall g \in \cD[\fh] \}
\]
and
\[
	\fh[ f , g ]
	=
	\langle Hf , g \rangle_{L^2(\RR^d)}
	,\quad
	f \in \cD(H),\ g \in \cD[\fh]
	.
\]
Since $V(x) \to \infty$ as $\abs{x} \to \infty$, it is well known that $H$ has purely discrete spectrum, see, e.\,g.,
\cite[Proposition~12.7]{Schmuedgen-12}. Moreover, a form core for $H$ is given by $C_c^\infty(\RR^d)$, see, e.\,g.,
\cite[Theorem~1.13]{CyconFKS-87}, that is, every function in $\cD[\fh]$ can be approximated in the form norm for $H$ by functions
in $C_c^\infty(\RR^d)$; a simple proof of this fact for the current type of potential (in particular,
$V \in L_\loc^\infty(\RR^d)$) can also be obtained from \cite[Lemma~2.2]{GagelmanY-12}.

Classic elliptic regularity results (see, e.\,g., \cite[Theorem~S2.2.1]{BerezinS-91}) imply that $\cD(H) \subset H_\loc^2(\RR^d)$
with $Hf = -\Delta f + Vf$ almost everywhere on $\RR^d$ for all $f \in \cD(H)$. In addition, if $Hf \in H_\loc^1(\RR^d)$ for some
$f \in \cD(H)$, then $f \in H_\loc^3(\RR^d)$.

The main objective of the present section is now to prove Theorem \ref{thm:decay-intro}.
In this course, if desired, the dependence of $C'$ in Theorem~\ref{thm:decay-intro} on the parameters $\tau_1,c_1,\nu, M_\nu$ can be traced explicitly
from the proof.
We refrain from doing so here for simplicity and brevity.

\subsection{Weighted inequalities}

We prove Theorem~\ref{thm:decay-intro} by establishing $L^2$-estimates for $f$ and $\abs{\nabla f}$ with an exponential weight. As a
preparation, we need the following technical lemma.

\begin{lemma}[see {\cite[Lemma~2.1]{GagelmanY-12}}]\label{lem:weak}
	Suppose that for some $\phi \in L^2(\RR^d)$ and $\lambda \geq 0$ the function $f \in H_\loc^2(\RR^d) \cap L^2(\RR^d)$
	satisfies $-\Delta f + Vf - \lambda f = \phi$ almost everywhere. Then, $f \in \cD[\fh]$ and for all $g \in \cD[\fh]$ we have
	\[
		\fh[ f , g ] - \lambda \langle f , g \rangle_{L^2(\RR^d)}
		=
		\langle \phi , g \rangle_{L^2(\RR^d)}
		.
	\]
\end{lemma}

\begin{proof}
	The fact that $f \in \cD[\fh]$ follows from
	\[
		\int_{\RR^d} (\abs{\nabla f}^2 + V\abs{f}^2)
		\leq
		\norm{ \phi }_{L^2(\RR^d)} \norm{ f }_{L^2(\RR^d)} + \lambda\norm{ f }_{L^2(\RR^d)}^2
		,
	\]
	which is proved verbatim as in \cite[Lemma~2.1]{GagelmanY-12}; the smoothness of the potential $V$ assumed there is actually not
	used and not needed.

	For $g \in C_c^\infty(\RR^d)$ we then obtain by integration by parts that
	\begin{align*}
		\fh[ f , g ]
		&=
		\fa[ f , g ] + \fv[ f , g ]
			=
			\int_{\RR^d} (-\Delta f + Vf)\overline{g}
			=
			\int_{\RR^d} (\phi + \lambda f)\overline{g}\\
		&=
		\langle \phi + \lambda f , g \rangle_{L^2(\RR^d)}
		,
	\end{align*}
	and the latter extends to all $g \in \cD[\fh]$ by approximation since
	$C_c^\infty(\RR^d)$ is a form core for $H$; cf.\ the discussion after Lemma~2.2 in \cite{GagelmanY-12}.
\end{proof}%

The next result is now at the core of our proof of Theorem~\ref{thm:decay-intro} and is a quantitative version of the statement in
\cite[Lemma~2.3]{GagelmanY-12}. Its proof is also extracted from that reference.

\begin{lemma}\label{lem:mainTech}
	Let $\lambda \geq 0$, $\mu > 0$, and $R \geq 1$ be such that $V(x) \geq \mu^2 + \lambda + 1$ whenever $\abs{x} \geq R$.
	Moreover, suppose that $f \in H_\loc^2(\RR^d)\cap L^2(\RR^d)$ satisfies $-\Delta f + Vf - \lambda f = \phi$ almost everywhere
	with some $\phi \in L^2(\RR^d)$. Then, if $\euler^{2\mu\abs{x}} \phi \in L^2(\RR^d)$, we have
	\be\label{eq:mainTech:expWeight}
		\norm{ \euler^{\mu\abs{x}} f }_{L^2(\RR^d)}^2
		\leq
		\frac{1}{2}\norm{ \euler^{2\mu\abs{x}}\phi }_{L^2(\RR^d\setminus B(0,R))}^2 +
			(4\mu + 6) \euler^{2\mu(R+1)}\norm{ f }_{L^2(\RR^d)}^2
		.
	\ee
\end{lemma}

\begin{proof}
	According to Lemma~\ref{lem:weak}, $f$ belongs to $\cD[\fh]$. We first suppose that $f$ is real-valued.
	Choose an infinitely differentiable function
	$\chi \colon \RR^d \to [0,1]$ with $\chi(x) = 0$ for
	$\abs{x} \leq R$ and $\chi(x) = 1$ for $\abs{x} \geq R+1$ such that
	$\norm{ \abs{ \nabla\chi } }_{L^\infty(\RR^d)} \leq 2$.
	For $\eps > 0$ let $w(x) = w_\eps(x) = \mu\abs{x}/(1+\eps\abs{x})$.
	Then $w$ is bounded and infinitely differentiable on $\RR^d\setminus\{0\}$.
	Accordingly, the same is true for $\chi\euler^w$ and $\chi\euler^{2w}$.
	Therefore, $\chi\euler^{2w}f$, $\chi^2\euler^{2w}f$, and
	$g := \chi\euler^w f$ are all real-valued, belong to $\cD[\fh]$, and vanish
	in the ball $B(0,R)$.
	In particular, the choice of
	$R$ implies that $\fv[ g ,g ]
	\geq (\mu^2 + \lambda + 1)\norm{ g }_{L^2(\RR^d)}^2$.
	Moreover, with the relation
	$\nabla(\euler^{\pm w}g) = \euler^{\pm w}\nabla g \pm g\euler^{\pm w}\nabla w$
	and the identity
	$\norm{ \abs{ \nabla w } }_{L^\infty(\RR^d\setminus\{0\})} = \mu$
	we obtain
	\[
		\nabla(\euler^{-w}g) \cdot \nabla(\euler^w g)
		=
		\abs{ \nabla g }^2 - \abs{ g }^2 \abs{ \nabla w }^2
		\geq
		-\mu^2\abs{ g }^2
		,
	\]
	so that
	\[
		\fh[ \chi f , \chi\euler^{2w}f ]
		=
		\fh[ \euler^{-w}g , \euler^w g ]
		=
		\fa[ \euler^{-w}g , \euler^w g ] + \fv[ g , g ]
		\geq
		(\lambda+1)\norm{ g }_{L^2(\RR^d)}^2
		.
	\]
	The latter can be rewritten as
	\be\label{eq:tech}
		\norm{ \chi\euler^w f }_{L^2(\RR^d)}^2
		\leq
		\fh[ \chi f , \chi\euler^{2w}f ] - \lambda\langle f , \chi^2\euler^{2w}f \rangle_{L^2(\RR^d)}
		.
	\ee

	Clearly, $\fv[ \chi f , \chi\euler^{2w}f ] = \fv[ f , \chi^2\euler^{2w}f ]$. Moreover, a straightforward computation shows that
	$\nabla(\chi f) \cdot \nabla(\chi\euler^{2w}f) = \nabla f \cdot \nabla(\chi^2\euler^{2w}f) + \eta\euler^{2w}\abs{f}^2$ with
	\be\label{eq:defEta}
		\eta
		:=
		2\chi \nabla\chi \cdot \nabla w + \abs{ \nabla \chi }^2
		.
	\ee
	Taking into account Lemma~\ref{lem:weak}, we therefore have
	\begin{align*}
		\fh[ \chi f , \chi\euler^{2w}f ]
		&=
		\fh[ f , \chi^2 \euler^{2w}f ] + \langle f , \eta\euler^{2w}f \rangle_{L^2(\RR^d)}\\
		&=
		\langle \phi + \lambda f, \chi^2 \euler^{2w}f \rangle_{L^2(\RR^d)} + \langle f , \eta\euler^{2w}f \rangle_{L^2(\RR^d)}
		.
	\end{align*}
	Plugging the latter into \eqref{eq:tech} gives
	\be\label{eq:cutWeightedBound}
		\begin{aligned}
			\norm{ \chi\euler^w f }_{L^2(\RR^d)}^2
			&\leq
			\langle \phi , \chi^2 \euler^{2w}f \rangle_{L^2(\RR^d)} + \langle f , \eta\euler^{2w}f \rangle_{L^2(\RR^d)}\\
			&=
			\langle \chi^2 \euler^{2w}\phi , f \rangle_{L^2(\RR^d)} + \langle f , \eta\euler^{2w}f \rangle_{L^2(\RR^d)}\\
			&\leq
			\norm{ \chi^2\euler^{2w}\phi }_{L^2(\RR^d)} \norm{ f }_{L^2(\RR^d)} + \norm{ \eta\euler^{2w} }_{L^\infty(\RR^d)}
				\norm{ f }_{L^2(\RR^d)}^2
			.
		\end{aligned}
	\ee
	The function $\eta$ in \eqref{eq:defEta} vanishes outside of the annulus
	$R < \abs{x} < R+1$ and satisfies
	$\abs{\eta} \leq 2\abs{ \nabla\chi }\abs{ \nabla w } +
	\abs{ \nabla\chi }^2 \leq 4(\mu + 1)$.	
	Hence,
	\[
		\norm{ \eta\euler^{2w} }_{L^\infty(\RR^d)}
		\leq
		4(\mu + 1) \euler^{2\mu(R+1)}
		.
	\]
	We thus conclude from \eqref{eq:cutWeightedBound} that
	\eqs{
		\norm{ \euler^{w} f }_{L^2(\RR^d)}^2
		&=
		\norm{ \euler^w f }_{L^2(B(0,R+1))}^2 + \norm{ \euler^w f }_{L^2(\RR^d\setminus B(0,R+1))}^2\\
		&\leq
		\euler^{2\mu(R+1)} \norm{ f }_{L^2(\RR^d)}^2 + \norm{ \chi\euler^w f }_{L^2(\RR^d)}^2\\
		&\leq
		\norm{ \chi^2\euler^{2w}\phi }_{L^2(\RR^d)}
		\norm{ f }_{L^2(\RR^d)} + (4\mu + 5)
		\euler^{2\mu(R+1)}
		\norm{ f }_{L^2(\RR^d)}^2
		\\
		&\leq
		\norm{ \euler^{2w}\phi }_{L^2(\RR^d \setminus B(0,R))}
		\norm{ f }_{L^2(\RR^d)} + (4\mu + 5)
		\euler^{2\mu(R+1)}
		\norm{ f }_{L^2(\RR^d)}^2\\
		&\leq
		\frac{1}{2}\norm{ \euler^{2w}\phi }_{L^2(\RR^d \setminus B(0,R))}^2 + (4\mu + 6)\euler^{2\mu(R+1)}\norm{ f }_{L^2(\RR^d)}^2
		,
	}
	where we used Young's inequality for the last estimate.
	Since $w(x) = w_\eps(x) \to \mu\abs{x}$ as $\eps \to 0$
	pointwise and monotonically, \eqref{eq:mainTech:expWeight} now follows by monotone
	convergence theorem.

	If $f$ is not real-valued, we proceed analogously for $\Re f$ and $\Im f$ separately and combine the obtained inequalities
	to arrive again at \eqref{eq:mainTech:expWeight}.
\end{proof}%

Applying Lemma~\ref{lem:mainTech} with $\phi = 0$ allows us to obtain the desired weighted $L^2$-estimates for
eigenfunctions of $H$, where $R$ can be computed from $\lambda$ and the constants
in part (i) of Assumption~\ref{ass:potential}.

\begin{proposition}\label{prop:weighted-est-eigenfunction}
	Suppose that $f \in \cD(H)$ with $Hf = \lambda f$ for some $\lambda \geq 0$, and choose $R\geq 1$ such that
	$R^{\tau_1} \geq (\lambda + 2)/c_1$. Then, we have
	\[
		\norm{\euler^{\abs{x}/2}f}_{L^2(\RR^d)}^2
		\leq
		8\euler^{R+1}\norm{f}_{L^2(\RR^d)}^2
		.
	\]
\end{proposition}

\begin{proof}
	According to the discussion in Subsection~\ref{ssec:operator}, the function $f$ belongs to $H_\loc^2(\RR^d)$ and satisfies
	$-\Delta f + Vf - \lambda f = 0$ almost everywhere. Applying Lemma~\ref{lem:mainTech} with $\mu = 1/2$ and $\phi = 0$ therefore
	proves the claim.
\end{proof}

In order to obtain by means of Lemma~\ref{lem:mainTech} an analogous result for the partial derivatives of an eigenfunction, we
follow the approach of \cite{GagelmanY-12} and differentiate the eigenvalue equation $Hf = \lambda f$. Indeed, since
$Hf \in H_\loc^2(\RR^d)$, we know that, in fact, $f$ belongs to $H_\loc^3(\RR^d)$, and it follows that each
$\partial_j f \in H_\loc^2(\RR^d)$, $j = 1,\dots,d$, satisfies
\be\label{eq:diffEq}
	-\Delta \partial_j f + V\partial_j f - \lambda \partial_j f
	=
	-f\partial_j V
\ee
almost everywhere. This allows to apply Lemma~\ref{lem:mainTech} to $\partial_j f$ with a corresponding right-hand side and, thus,
leads to the following result.

\begin{proposition}\label{prop:weighted-est-gradient-eigenfunction}
	Let $f \in \cD(H)$ with $Hf = \lambda f$ for some $\lambda \geq 0$, and choose $R \geq 1$ such that
	$R^{\tau_1} \geq ((\nu+1)^2 + \lambda + 1)/c_1$. Then, we have
	\[
		\norm{ \euler^{\abs{x}/2} \abs{ \nabla f } }_{L^2(\RR^d)}^2
		\leq
		\bigl( 8\lambda + (2\nu+5)M_\nu^2 \bigr) \euler^{2(1+\nu)(R+1)}
		\norm{ f }_{L^2(\RR^d)}^2
		.
	\]
\end{proposition}

\begin{proof}
	Denote by $\phi_j := -f\partial_j V$ the right-hand side of \eqref{eq:diffEq}.

	In light of the hypothesis on $R$, we may first apply Lemma~\ref{lem:mainTech} to $f$ with $\mu = \nu + 1$ and $\phi = 0$ to
	obtain
	\be\label{eq:expWeight}
		\norm{ \euler^{(1+\nu)\abs{x}}f }_{L^2(\RR^d)}^2
		\leq
		(4\nu+10) \euler^{2(1+\nu)(R+1)} \norm{ f }_{L^2(\RR^d)}^2
		.
	\ee
	Since $\abs{ \phi_j(x) } \leq M_\nu \euler^{\nu\abs{x}} \abs{ f }$ on $\RR^d\setminus B(0,1)$, we conclude that
	$\euler^{\abs{x}} \phi_j \in L^2(\RR^d\setminus B(0,1))$.
	In view of \eqref{eq:diffEq}, we may then again apply Lemma~\ref{lem:mainTech}, this time to $\partial_j f$
	with $\mu = 1/2$ and $\phi = \phi_j = -f\partial_j V$, which gives
	\be
		\norm{ \euler^{\abs{x}/2} \partial_j f }_{L^2(\RR^d)}^2
		\leq
		\frac{1}{2}\norm{ \euler^{\abs{x}}\phi_j }_{L^2(\RR^d\setminus B(0,1))}^2 + 8\euler^{R+1} \norm{ \partial_j f }_{L^2(\RR^d)}^2
		.
	\ee
	Taking into account \eqref{eq:expWeight} and that
	\[
		\norm{ \abs{ \nabla f } }_{L^2(\RR^d)}^2
		=
		\fa[ f , f ]
		\leq
		\fh[ f , f ]
		=
		\langle Hf , f \rangle_{L^2(\RR^d)}
		=
		\lambda \norm{ f }_{L^2(\RR^d)}^2
		,
	\]
	summing over $j$ then yields
	\begin{align*}
		\norm{ \euler^{\abs{x}/2} \abs{ \nabla f } }_{L^2(\RR^d)}^2
		&\leq
		\frac{1}{2}\norm{ \euler^{\abs{x}}f\abs{ \nabla V } }_{L^2(\RR^d\setminus B(0,1))}^2 + 8\euler^{R+1} \norm{ \abs{ \nabla f } }_{L^2(\RR^d)}^2\\
		&\leq
		\frac{M_\nu^2}{2}\norm{ \euler^{(1+\nu)\abs{x}}f }_{L^2(\RR^d\setminus B(0,1))}^2 + 8\lambda\euler^{R+1} \norm{ f }_{L^2(\RR^d)}^2\\
		&\leq
		\bigl( 8\lambda + (2\nu+5)M_\nu^2 \bigr) \euler^{2(1+\nu)(R+1)} \norm{ f }_{L^2(\RR^d)}^2
		,
	\end{align*}
	which proves the claim.
\end{proof}%

\subsection{Proof of Theorem~\ref{thm:decay-intro}}
Recall that $H$ has purely discrete spectrum, and let $(\lambda_k)_{k\in\NN}$ be an enumeration of its spectrum $\sigma(H)$ in
nondecreasing order (without multiplicities). With
\[
	N(\lambda)
	:=
	\#(\sigma(H)\cap (-\infty,\lambda])
	,
\]
we may then expand every $f \in \Ran P_\lambda(H)$ as
\be\label{eq:fExpand}
	f
	=
	\sum_{k=1}^{N(\lambda)} f_k
\ee
where $f_k=\indic_{\{\lambda_k\}}(H)f$ for $k \in \{1,\dots,N(\lambda)\}$.
Note that we have the simple bound
\[
	N(\lambda)
	\leq
	\#\{ k \colon \lambda_k \leq \lambda \}
		\leq
		\sum_{k \colon \lambda_k \leq \lambda} (\lambda + 1 - \lambda_k )
		\leq
		\sum_{k \colon \lambda_k \leq \lambda + 1} (\lambda + 1 - \lambda_k )
\]
and in light of the lower bound $V(x) \geq c_1\abs{x}^{\tau_1}$ on the potential in part (i) of
Assumption~\ref{ass:potential}, the right hand side can be estimated explicitly by means of the classic
Lieb-Thirring bound from \cite[Theorem~1]{LiebT-91}. More precisely, for $\lambda \geq 1$ we have
\eqs{
	\sum_{k \colon \lambda_k \leq \lambda + 1} (\lambda + 1 - \lambda_k )
	&\lesssim_d
	\int_{\RR^d} \max\{ \lambda + 1 - V(x) , 0 \}^{d/2 + 1} \Diff{x} \\
	&\leq
	\int_{B(0,((\lambda+1)/c_1)^{1/\tau_1})} (\lambda+1)^{d/2 + 1} \Diff{x} \\
	&\lesssim_{d,\tau_1,c_1}
	\lambda^{1 + d(1/2+1/\tau_1)}
	,
}
and, therefore,
\be\label{eq:boundN}
	N(\lambda)
	\lesssim_{d,\tau_1,c_1}
	\lambda^{1 + d(1/2+1/\tau_1)}
	.
\ee

\begin{remark}
	Note that the Lieb-Thirring bound actually also takes into account multiplicities.
	It is worth to mention that for $d \geq 3$ the classic Cwikel-Lieb-Rozenblum bound provides a sharper bound on
	$N(\lambda)$, but the above is more than sufficient for our purposes.
\end{remark}

We are now in position to prove the main result of this section.

\begin{proof}[Proof of Theorem~\ref{thm:decay-intro}]
	For every $r > 0$, we have
	\eqs{
		\norm{ f }_{H^1(\RR^d\setminus B(0,r))}^2
		&=
		\norm{ f }_{L^2(\RR^d\setminus B(0,r))}^2 + \norm{ \abs{\nabla f} }_{L^2(\RR^d\setminus B(0,r))}^2\\
		&\leq
		\euler^{-r} \bigl( \norm{ \euler^{\abs{x}/2}f }_{L^2(\RR^d)}^2 +
			\norm{ \euler^{\abs{x}/2}\abs{\nabla f} }_{L^2(\RR^d)}^2 \bigr)
		.
	}
	Moreover, using the expansion \eqref{eq:fExpand} and H\"older's inequality, we may estimate
	\[
	  \norm{ \euler^{\abs{x}/2}f }_{L^2(\RR^d)}^2
		\leq
		\Biggl( \sum_{k=1}^{N(\lambda)} \norm{ \euler^{\abs{x}/2}f_k }_{L^2(\RR^d)} \Biggr)^2
		\leq
		N(\lambda) \sum_{k=1}^{N(\lambda)} \norm{ \euler^{\abs{x}/2}f_k }_{L^2(\RR^d)}^2
	\]
	and similarly, taking into accout $\abs{ \nabla f } \leq \sum_{k=1}^{N(\lambda)} \abs{ \nabla f_k }$,
	\[
		\norm{ \euler^{\abs{x}/2}\abs{ \nabla f } }_{L^2(\RR^d)}^2
		\leq
		N(\lambda) \sum_{k=1}^{N(\lambda)} \norm{ \euler^{\abs{x}/2}\abs{ \nabla f_k } }_{L^2(\RR^d)}^2
		.
	\]

	We choose $R := ( (\nu+1)^2 + \lambda + 1 )^{1/\tau_1}/c_1 \lesssim_{\nu,\tau_1, c_1} \lambda^{1/\tau_1}$, which meets the requirement on
	$R$ in both Propositions~\ref{prop:weighted-est-eigenfunction} and~\ref{prop:weighted-est-gradient-eigenfunction} for all
	eigenfunctions corresponding to eigenvalues not exceeding $\lambda$. In particular, this is the case for the functions $f_k$ in
	the expansion \eqref{eq:fExpand}.
	Since $\sum_{k=1}^{N(\lambda)} \norm{f_k}_{L^2(\RR^d)}^2 = \norm{f}_{L^2(\RR^d)}^2$ and in light of \eqref{eq:boundN}, applying
	Propositions~\ref{prop:weighted-est-eigenfunction} and~\ref{prop:weighted-est-gradient-eigenfunction} for each $f_k$ separately
	therefore implies that there is a constant $\tilde{C} > 0$, depending only on $c_1,\tau_1,\nu,M_\nu$, and $d$, such that
	\[
		\norm{ f }_{H^1(\RR^d\setminus B(0,r))}^2
		\leq
		\euler^{-r} \euler^{\tilde{C}\lambda^{1/\tau_1}} \norm{f}_{L^2(\RR^d)}^2
		.
	\]
	Choosing $r := \log 2 + \tilde{C}\lambda^{1/\tau_1} \leq (\tilde{C}+\log 2)\lambda^{1/\tau_1}$ then proves the claim with the
	constant $C' = \tilde{C} + \log 2$.
\end{proof}

%
%

\section{Proof of the spectral inequality}\label{sec:proof}
Due to Theorem~\ref{thm:decay-intro} it is sufficient to derive an analog of the spectral inequality on a large, but finite cube.
Once the potential $V$ is restricted to the cube, it is a bounded function. This is the situation for which
Section 3 of \cite{NakicTTV-18} provides appropriate interpolation and spectral inequality from equidistributed sets.
Note that \cite{NakicTTV-20a} extends these results to unbounded domains including $\RR^d$, which is of relevance to us
since the operator $H$ is defined on $\RR^d$.

Our proof of Theorem~\ref{thm:spectral-inequality} relies on an adaptation of these estimates, which we present next.

\subsection{Ghost dimension}\label{ssec:ghost-dimension}

We make use of the so-called \emph{ghost dimension} construction, which was first introduced in \cite{JerisonL-99}. Following the
proofs in \cite{NakicTTV-18,NakicTTV-20a,DickeRST-20}, we denote by $(\cF_t)_{t\in\RR}$ the family of unbounded selfadjoint
operators
\[
	\cF_t
	=
	\int_\RR s_t(\mu) \,\Diff{P_\mu(H)}
	,\quad
	s_t(\mu)
	=
	\begin{cases}
		\frac{\sinh(\sqrt{\mu}t)}{\sqrt{\mu}} ,& \mu > 0,\\
		t ,& \mu = 0,
	\end{cases}
\]
in $L^2(\RR^d)$. For fixed $f \in \Ran P_\lambda(H)$, $\lambda \geq 0$, we then define $F \colon \RR^d \times \RR \to \CC$ by
\be\label{def:ghost}
	F(\cdot,t)
	=
	\cF_t f
	\in
	\Ran P_\lambda(H)
	\subset
	\cD(H)
	.
\ee
Expanding $f$ as in \eqref{eq:fExpand} we clearly have
\be\label{eq:ghostExpand}
	F(x,t)
	=
	\sum_{k=1}^{N(\lambda)} f_k(x)s_t(\lambda_k)
	,\quad
	(x,t) \in \RR^d \times \RR
	.
\ee
From this, we easily see that $F$ is measurable and belongs to $H_\loc^2(\RR^{d+1})$. Moreover, we clearly have
$\partial_t F(\cdot,t) \in \Ran P_\lambda(H)$.
Taking into account that $\partial_t s_t(\mu)|_{t=0}=1$ and $\partial_t^2 s_t(\mu)=\mu s_t(\mu)$
for all $\mu \geq 0$, it also follows that
\be\label{eq:ghostBoundary}
	(\partial_t F)(\cdot,0)
	=
	f
	,
\ee
as well as
\be\label{eq:ghostHarmonic}
	H(F(\cdot,t))
	=
	(\partial_t^2 F)(\cdot,t)
	\quad\text{ for all }\quad
	t \in \RR
	.
\ee

The following lemma is an analogue to \cite[Proposition~3.6]{NakicTTV-18},
\cite[Proposition~2.9]{NakicTTV-20a}, and \cite[Lemma~6.1]{DickeRST-20} and connects
the extended function $F$ to the original function $f$.

\begin{lemma}\label{lem:ghost}
	Let $f \in \Ran P_\lambda(H)$, and let $F \colon \RR^d \times \RR \to \CC$ be defined as in \eqref{def:ghost}. Then,
	for every $\varrho > 0$ the restriction of $F$ to $\RR^d \times (-\varrho,\varrho)$ belongs to the Sobolev space
	$H^1(\RR^d \times (-\varrho,\varrho))$ with
	\[
		2\varrho\norm{ f }_{L^2(\RR^d)}^2
		\leq
		\norm{ F }_{H^1(\RR^d \times (-\varrho,\varrho))}^2
		\leq
		2\varrho(1+(1+\lambda) \varrho^2)\euler^{2\varrho\sqrt{\lambda}} \norm{ f }_{L^2(\RR^d)}^2
		.
	\]
\end{lemma}

\begin{proof}
	Using $\abs{ \sinh(\sqrt{\mu}t) } \leq \abs{t}\sqrt{\mu}\cosh(\sqrt{\mu}t)$,
	$1 \leq \cosh(\sqrt{\mu}t) \leq \euler^{\abs{t}\sqrt{\mu}}$, and the identity $\partial_t s_t(\mu) = \cosh(\sqrt{\mu}t)$ for all
	$t \in \RR$ and $\mu \geq 0$, we easily obtain from the expansion	\eqref{eq:ghostExpand} that
	\be\label{eq:FboundUpper}
		\norm{ F(\cdot,t) }_{L^2(\RR^d)}^2
		\leq
		t^2\euler^{2\abs{t}\sqrt{\lambda}} \norm{ f }_{L^2(\RR^d)}^2
		\leq
		\varrho^2\euler^{2\varrho\sqrt{\lambda}} \norm{ f }_{L^2(\RR^d)}^2
		,
	\ee
	as well as
	\be\label{eq:derFbounds}
		\norm{ f }_{L^2(\RR^d)}^2
		\leq
		\norm{ \partial_t F(\cdot,t) }_{L^2(\RR^d)}^2
		\leq
		\euler^{2\abs{t}\sqrt{\lambda}} \norm{ f }_{L^2(\RR^d)}^2
		\leq
		\euler^{2\varrho\sqrt{\lambda}} \norm{ f }_{L^2(\RR^d)}^2
	\ee
	for all $t \in (-\varrho,\varrho)$.
	
	Taking into account that $F(\cdot,t) \in \Ran P_\lambda(H) \subset \cD(H)$ for all $t \in \RR$, we now clearly have
	\begin{align*}
		\sum_{k=1}^d \norm{ \partial_k F(\cdot,t) }_{L^2(\RR^d)}^2
		&\leq
		\langle H F(\cdot,t) , F(\cdot,t) \rangle_{L^2(\RR^d)}
			\leq
			\lambda \norm{ F(\cdot,t) }_{L^2(\RR^d)}^2\\
		&\leq
		\lambda \varrho^2\euler^{2\varrho\sqrt{\lambda}} \norm{ f }_{L^2(\RR^d)}^2
		.
	\end{align*}
	Combining the latter with \eqref{eq:FboundUpper} and the upper bound in \eqref{eq:derFbounds} and integrating over
	$t \in (-\varrho,\varrho)$ proves that $F|_{\RR^d \times (-\varrho,\varrho)}$ belongs to $H^1(\RR^d\times(-\varrho,\varrho))$
	satisfying the upper bound in the claim.

	For the lower bound, we simply observe that
	\[
		\norm{ F }_{H^1(\RR^d \times (-\varrho,\varrho))}^2
		\geq
		\int_{-\varrho}^\varrho \norm{ \partial_t F(\cdot,t) }_{L^2(\RR^d)}^2 \Diff{t}
		\geq
		2\varrho \norm{ f }_{L^2(\RR^d)}^2
	\]
	by the lower bound in \eqref{eq:derFbounds}, which completes the proof.
\end{proof}%

\subsection{Proof of Theorem~\ref{thm:spectral-inequality}}

Let $\lambda \geq 1$ and $f\in\Ran P_\lambda(H) \setminus \{0\}$, and define $F$ as in \eqref{eq:ghostExpand}. We infer from
Theorem~\ref{thm:decay-intro} that there is a constant $C' > 0$, depending on $\tau_1,c_1,\nu,M_\nu$, and $d$, such that
\[
	\norm{ g }_{H^1(\RR^d)}^2
	\leq
	2\norm{ g }_{H^1(B(0,C'\lambda^{1/\tau_1}))}^2
	\quad\text{ and }\quad
	\norm{ g }_{L^2(\RR^d)}^2
	\leq
	2\norm{ g }_{L^2(B(0,C'\lambda^{1/\tau_1}))}^2
\]
for all $g \in \Ran P_\lambda(H)$. Applying the latter for each $t \in (-1,1)$ to $g = F(\cdot,t)$ and
$g = \partial_t F(\cdot,t)$, respectively, yields
\be\label{eq:decay-gamma-potential-ghost}
	\norm{F}_{H^1(\RR^d\times(-1,1))}^2
	\leq
	2\norm{F}_{H^1(B(0,C'\lambda^{1/\tau_1})\times(-1,1))}^2
	.
\ee

Let $\Lambda$ be the smallest cube of integer sidelength centered at the origin that contains $B(0,C'\lambda^{1/\tau_1})$.
For technical reasons, we from now on suppose that $C' \geq 5$, so that $\Lambda$ has sidelength at least $5$.
Hence $\Lambda:= (-L,L)^d$ with $2L\in \{5,6,7,\ldots\}$ and $C'\lambda^{1/\tau_1} \leq L \leq C'\lambda^{1/\tau_1} +1\leq 2C'\lambda^{1/\tau_1} $.
Set
$\cK := \cK(\lambda) := \{ k \in \ZZ^d \colon k \in \Lambda \}$. 
Then, $|k|\leq \sqrt{d}L$, hence $1+|k|^\alpha\leq 2 d^{\alpha/2}L^\alpha$ for all
$k \in \cK$, so that
\be\label{eq:defEquiconst}
	\delta^{1+\abs{k}^\alpha}
	\geq
	\delta^{2(2\sqrt{d}C')^\alpha \lambda^{\alpha/\tau_1}}
	=:
	\theta
	\quad\text{ for all }\quad
	k \in \cK
	.
\ee
Moreover, the closure of $\Lambda$ agrees with the union $\bigcup_{k\in\cK} (k + [-1/2,+1/2]^d)$, and the hypothesis on $\omega$
implies that for each $k \in \cK$ the intersection $\omega \cap (k + (-1/2,+1/2)^d)$ contains a ball of radius $\theta$. In
particular, $\omega \cap \Lambda$ is $\theta$-equidistributed (in $\Lambda$) in the sense of
\cite{NakicTTV-18,NakicTTV-20a,DickeRST-20}.

The next interpolation result is a consequence of
the Carleman estimates from \cite{NakicRT-19, LebeauR-95, JerisonL-99} and
relies on the fact that by the upper bound in part (i) of Assumption~\ref{ass:potential}
the potential $V$ can be bounded on $\Lambda$ (or a suitably scaled version thereof) by
a multiple of $(\lambda^{1/\tau_1})^{\tau_2}$. The situation therefore morally boils down to the one of
bounded potentials, so that the proof can be extracted from \cite{NakicTTV-18,NakicTTV-20a} and is omitted here.
We refer, however, to \cite[Proposition~6.14]{Dicke-22} for a detailed presentation.

\begin{proposition}\label{prop:covering-interpolation}
	Let $\theta$ be as in \eqref{eq:defEquiconst}, and set $R := 9\euler\sqrt{d}$. Then, there is $\kappa \in (0,1)$ satisfying
	$1/\kappa \lesssim_d \log(1/\theta)$ and a constant $D \geq 1$, depending on $c_2,\tau_2,C'$, and the dimension $d$, such
	that
	\[
		\norm{F}_{H^1(\Lambda\times(-1,1))}
		\leq
		\theta^{-\kappa D\lambda^{2\tau_2/3\tau_1}}
			\norm{F}_{H^1(\RR^d\times(-R,R))}^{1-\kappa/2}
			\norm{f}_{L^2(\omega\cap\Lambda)}^{\kappa/2}
		.
	\]
\end{proposition}

\begin{remark}\label{rem:Carleman}
	The appearance of the power $\lambda^{2\tau_2/3\tau_1}$ in Proposition~\ref{prop:covering-interpolation} is due to the following:
Once we insert in the upper bound $\|V\|_{\infty,\Lambda}\leq  c_2 d^{\tau_2/2} \, L^{\tau_2}$ on the potential on the box $\Lambda:= (-L,L)^d$
the bound on the sidelength $L \leq C'\lambda^{1/\tau_1} +1$,
we obtain $\|V\|_{\infty,\Lambda}\lesssim \lambda^{\tau_2/\tau_1}$.
The sup of the potential enters the interpolation inequality via the Carleman estimate
and appears in the scaling $\|V\|_{\infty,\Lambda}^{2/3}$, giving rise to the term  $\lambda^{2\tau_2/3\tau_1}$.
\end{remark}

With Proposition~\ref{prop:covering-interpolation} at hand, we are finally in position to prove
the main result of this note.

\begin{proof}[Proof of Theorem~\ref{thm:spectral-inequality}]
	We adopt the notation established in the preceding part of the current section.

	The lower bound in Lemma~\ref{lem:ghost} for $\varrho = R = 9\euler\sqrt{d}$ gives
	\[
		\norm{ f }_{L^2(\RR^d)}^2
		\leq
		\frac{1}{2R} \norm{ F }_{H^1(\RR^d\times(-R,R))}^2
		.
	\]
	In order to estimate the right-hand side further,  we combine the lower bound in Lemma~\ref{lem:ghost} for $\varrho = 1$ with
	the corresponding upper bound for $\varrho = R$ and obtain
	\[
		\frac{\norm{F}_{H^1(\RR^d\times(-R,R))}^2}{\norm{F}_{H^1(\RR^d\times(-1,1))}^2}
		\leq
		R(1+(1+\lambda) R^2)\euler^{2R\sqrt{\lambda}}
		\leq
		\euler^{K \lambda^{1/2}}
	\]
	with some constant $K$ depending only on the dimension. Together with \eqref{eq:decay-gamma-potential-ghost} and the
	bound from Proposition~\ref{prop:covering-interpolation} this yields
	\eqs{
		\norm{ F }_{H^1(\RR^d\times(-R,R))}^2
		&\leq
		\euler^{K\lambda^{1/2}}\norm{ F }_{H^1(\RR^d\times(-1,1))}^2
			\leq
			2\euler^{K\lambda^{1/2}} \norm{ F }_{H^1(\Lambda\times(-1,1))}^2\\
		&\leq
		\euler^{2K\lambda^{1/2}} \theta^{-2\kappa D \lambda^{2\tau_2/3\tau_1}} \norm{ F }_{H^1(\RR^d\times(-R,R))}^{2-\kappa}
			\norm{ f }_{L^2(\omega\cap\Lambda)}^\kappa
		,
	}
	that is,
	\[
		\norm{ F }_{H^1(\RR^d\times(-R,R))}
		\leq
		\euler^{2K\lambda^{1/2}/\kappa} \theta^{-2D \lambda^{2\tau_2/3\tau_1}} \norm{ f }_{L^2(\omega\cap\Lambda)}
		.
	\]
	In light of $1/\kappa \lesssim_d \log(1/\theta)$, $2\tau_2/3\tau_1 \geq 2/3 > 1/2$, and the definition of $\theta$ in
	\eqref{eq:defEquiconst}, the claim follows upon a suitable choice of the constant $C$ depending on $C', D$, and the dimension
	$d$.
\end{proof}%

\section{Partial confinement potentials}
\label{sec:tensor}
The above proof can easily be adapted to more general anisotropic potentials $V$, such as certain partial confinement potentials.
By the latter we mean potentials that behave like in Assumption~\ref{ass:potential} only with respect to certain coordinate
directions. For simplicity, we demonstrate this in the case
\be\label{eq:partial-confinement-operator}
	H = -\Delta + V
	\quad \text{with} \quad
	V(x_1,x_2) = W(x_1)
	,\
	(x_1,x_2) \in \RR^{d_1} \times \RR^{d-d_1}
	,
\ee
where $d_1 \in \NN$, $d_1 < d$, and where $W \in W_\loc^{1,\infty}(\RR^{d_1})$ satisfies Assumption~\ref{ass:potential} with $d$
replaced by $d_1$.

Since the operator $H$ no longer has purely discrete spectrum, an expansion as in \eqref{eq:ghostExpand} for the extension $F$ of
$f \in \Ran P_\lambda(H)$ via the ghost dimension construction is not available. However, straightforward adaptations of the
arguments in \cite{NakicTTV-20a,DickeRST-20} show that we still have $F \in H_\loc^2(\RR^{d+1})$,
$F(\cdot,t),\partial_t F(\cdot,t) \in \Ran P_\lambda(H)$ for all $t \in \RR$, as well as \eqref{eq:ghostBoundary} and
\eqref{eq:ghostHarmonic}. Also an analogue to Lemma~\ref{lem:ghost} remains valid verbatim.

Following the reasoning in the proof of \cite[Lemma~A.2]{DickeSV-22}, we see that $H$ admits the tensor representation
\[
	H
	=
	H_1\otimes I_2 + I_1 \otimes H_2
	,
\]
where $H_1 = -\Delta + W$ in $L^2(\RR^{d_1})$ and $H_2 = -\Delta$ in $L^2(\RR^{d-d_1})$, and where $I_1$ and $I_2$ denote the
identity operators in $L^2(\RR^{d_1})$ and $L^2(\RR^{d-d_1})$, respectively. Consequently, as in \cite[Corollary~A.5]{DickeSV-22}
(cf.\ also the proof of \cite[Lemma~2.3]{EgidiS-21}), every function $h \in \Ran P_\lambda(H)$ can be written as a finite sum
\[
	h
	=
	\sum_{k} \phi_k \otimes \psi_k
\]
with suitable $\phi_k\in\Ran P_\lambda(H_1)$ and $\psi_k\in\Ran P_\lambda(H_2)$. Applying this to $h = F(\cdot,t)$ and
$\partial_t F(\cdot,t)$ implies that $\partial_t F(\cdot,x_2,t)$ and $\partial_{x_2}^\alpha F(\cdot,x_2,t)$,
$\abs{\alpha} \leq 1$, belong to $\Ran P_{\lambda}(H_1)$ for all $t \in \RR$ and (almost) all $x_2 \in \RR^{d-d_1}$.
Theorem~\ref{thm:decay-intro} for $H_1$ instead of $H$ therefore provides an analogue to \eqref{eq:decay-gamma-potential-ghost} with
respect to the first $d_1$ coordinates, that is,
\be
	\norm{F}_{H^1(\RR^d\times(-1,1))}^2
	\leq
	2\norm{F}_{H^1(B^{(d_1)}(0,C'\lambda^{1/\tau_1})\times\RR^{d-d_1}\times(-1,1))}^2
	,
\ee
where $B^{(d_1)}(0,C'\lambda^{1/\tau_1}) \subset \RR^{d_1}$.
After that, we may follow our proof of
Theorem~\ref{thm:spectral-inequality} verbatim to get a statement analogous to Theorem~\ref{thm:spectral-inequality},
which also covers Theorem~\ref{thm:anisotropic_V_easy} above.
We here just give the corresponding result and refer to \cite{Dicke-22} for more details.

\begin{theorem}\label{thm:anisotropic_V}
	Let $H$ be as in \eqref{eq:partial-confinement-operator}, and let $\omega \in \RR^d$ be measurable such that for some
	$\delta \in (0,1/2)$ and $\alpha \geq 0$ each intersection $\omega \cap (k + (-1/2,1/2)^d)$, $k = (k_1,k_2) \in \ZZ^d$, contains
	a ball of radius $\delta^{1+\abs{k_1}^\alpha}$.
	Then there is a constant $C > 0$ depending only on the dimension $d$ and the parameters $\tau_1, \tau_2, c_1, c_2, \nu, M_\nu$
	connected to $W$ such that for all $\lambda \geq 1$ and all $f \in \Ran P_\lambda(H)$ we have
	\bes
		\norm{f}_{L^2(\RR^d)}
		\leq
		\Bigl(\frac{1}{\delta}\Bigr)^{C^{1+\alpha}\cdot \lambda^{(\alpha+2\tau_2/3)/\tau_1}}
		\norm{f}_{L^2(\omega)}
		.
	\ees
\end{theorem}

Note that this theorem relates to the spectral inequality for the partial harmonic oscillator from \cite{DickeSV-22} in
the same way as Theorem~\ref{thm:spectral-inequality} relates to the one for the harmonic oscillator from \cite{DickeSV-23}.

\small
\section{Statements and declarations}
\noindent
Our manuscript has no associated data.
\\
There are no conflicts of interest related to this work.
\\
No animals, plants, or mushrooms have been harmed in the course of research for this manuscript, except for nourishment of the authors.

\normalsize

\newcommand{\etalchar}[1]{$^{#1}$}

\end{document}